\documentclass[a4paper,12pt,reqno]{amsart}
\usepackage{latexsym}
\usepackage{amsmath,amsthm,amssymb,amscd}
\usepackage{fullpage}
\usepackage{xcolor}
\usepackage{parskip}
\usepackage{mathtools}
\usepackage{thmtools}
\mathtoolsset{showonlyrefs}
\usepackage{csquotes}
\usepackage[activate={true,nocompatibility},final,tracking=true,kerning=true,spacing=true]{microtype}
\usepackage[pdfencoding=unicode,pdftex,bookmarks=true,bookmarksnumbered]{hyperref}
\definecolor{citecolour}{rgb}{0.0, 0.0, 0.8}
\colorlet{linkcolour}{green!50!black}
\hypersetup{colorlinks,breaklinks,
		linkcolor=linkcolour,citecolor=citecolour,
		filecolor=linkcolour, urlcolor=linkcolour}
\usepackage{version}

\newtheorem{prevtheorem}{Theorem}

\newtheorem{theorem}{Theorem}
\newtheorem{proposition}[theorem]{Proposition}

\newtheorem{lemma}[theorem]{Lemma}
\newtheorem{corollary}[theorem]{Corollary}

\newtheorem{question}{Question}

\newtheorem{example}[theorem]{Example}

\newtheorem{remark}[theorem]{Remark}

\numberwithin{equation}{section}
\usepackage[shortlabels]{enumitem}
\begingroup
    \makeatletter
    \@for\theoremstyle:=definition,remark,plain\do{%
        \expandafter\g@addto@macro\csname th@\theoremstyle\endcsname{%
            \addtolength\thm@preskip\parskip
            }%
        }
\DeclareMathOperator{\FF}{\mathfrak{F}}

\DeclareMathOperator{\N}{\mathbf{N}}

\DeclareMathOperator{\fitt}{\mathbf{F}}
\DeclareMathOperator{\Hall}{Hall}
\DeclareMathOperator{\cart}{Cart}
\DeclareMathOperator{\inj}{Inj}
\DeclareMathOperator{\cent}{\mathbf{C}}

\DeclareMathOperator{\oh}{\mathbf{O}}
\DeclareMathOperator{\syl}{\mathbf{Syl}}
\DeclareMathOperator{\aut}{\mathbf{Aut}}

\newcommand{\gensub}[1]{\left\langle{#1}\right\rangle}
\newcommand{\card}[1]{\left\lvert{#1}\right\rvert}
\renewcommand{\leq}{\leqslant}
\renewcommand{\geq}{\geqslant}
\newcommand{\md}[1]{\,\left(\textnormal{mod}\ #1\right)}

\newenvironment{proofofA}{{\bf {Proof of Theorem \ref{Thm:A}.} }}{\hfill $\blacksquare$ \\}
\newenvironment{proofofB}{{\bf {Proof of Theorem \ref{Thm:B}.} }}{\hfill $\blacksquare$ \\}
\newenvironment{proofofC}{{\bf {Proof of Theorem \ref{Thm:injectorratio}.} }}{\hfill $\blacksquare$ \\}
\newenvironment{proofof}{{\bf {Proof.} }}{\hfill $\blacksquare$ \\}

\def\NN{{\mathfrak N}}
\def\NS{\theta_\NN^G(1)}
\def\SS{{\mathcal S}}
\def\C_p{{\mathbf{C}(p)}}

\def\ff{{\mathbb F}}

\def\cl{{\mathsf{c}\mathsf{l}}}

\begin{document}
\title{Counting in nilpotent injectors and Carter subgroups}

\author{S. Aivazidis}
\address{Department of Mathematics \& Applied Mathematics, University of Crete, Greece}
\email{s.aivazidis@uoc.gr}
\author{M. Loukaki}
\address{Department of Mathematics \& Applied Mathematics, University of Crete, Greece}
\email{mloukaki@uoc.gr}
\author{J. Shareshian}
\address{Department of Mathematics, Washington University, One Brookings Drive, St Louis, MO 63124 USA}
\email{shareshi@math.wustl.edu}

 \thanks{The first author is partially supported by the Hellenic Foundation for Research and Innovation, Project HFRI-FM17-1733.}

\begin{abstract}
We investigate number-theoretic properties of the collection of nilpotent injectors or nilpotent projectors containing certain subgroups of finite soluble (or ${\mathcal N}$-constrained) groups. 
\end{abstract}
\maketitle

\section{Introduction} 

Let $G$ be a finite group and let $\pi$ be a set of primes. Divisibility results about the number of nilpotent groups 
from a characteristic conjugacy class that contain a given subgroup of a group $G$ appear throughout the literature.  
For example, Iranzo, Medina, and P\'erez Monasor proved in~\cite{arithmetical_questions} the following theorem, 
which generalises work of Navarro in~\cite{Navarro} and Turull in~\cite{Turull}:
\begin{quote}
Given any $\pi$-separable group $G$ and any $\pi$-subgroup $X \leq G$, 
write $\nu_\pi(G,X)$ for the number of Hall $\pi$-subgroups of $G$ containing $X$. 
If $X \leq K \leq G$, then $\nu_\pi(K,X)$ divides $\nu_\pi(G,1)$.
\end{quote}

A primary aim of this paper is to provide results of a similar nature in which $\pi$-subgroups are replaced with nilpotent subgroups 
and Hall $\pi$-subgroups are replaced with one of two conjugacy classes of subgroups:
\begin{itemize}
\item Recall that a group $G$ is said to be $\mathcal{N}$-constrained
provided that $\cent_G(\fitt(G)) \leq \fitt(G)$. 
It is a standard fact that all finite soluble groups are $\mathcal{N}$-constrained.
Moreover, a subgroup $I$ is called a \textbf{nilpotent injector} of $G$
provided that for each subnormal subgroup $S$ of $G$,
$S \cap I$ is a maximal nilpotent subgroup of $S$.  In any finite soluble group,
nilpotent injectors exist and comprise a single conjugacy class,
as was proved by Fischer, Gasch\"{u}tz and Hartley in~\cite{injectoren}.
In fact, by work of Mann~\cite{mann_injectors} 
all finite $\mathcal{N}$-constrained groups contain nilpotent injectors; 
again, such injectors comprise a single conjugacy class.

\item The dual notion of a nilpotent injector is that of a \textbf{Carter subgroup}. 
These are, by definition, nilpotent self-normalising subgroups of $G$ and  were
introduced by Carter in~\cite{Carter}. 
They were later seen to be  specific instances of something more general;
they are the so-called projectors for the class of nilpotent groups.
\end{itemize}

The two classes just mentioned arise from two separate theories:
Fitting classes and the associated injectors and saturated formations
and the associated projectors. 
In the case of $\pi$-subgroups the two notions coincide since the
class of $\pi$-groups is a saturated Fitting formation,
but in general injectors and projectors need not be the same (when they exist) and, indeed, this is the case
with the class of nilpotent groups.

Let $G$ be an $\mathcal{N}$-constrained group  and  $H\leq G$  a fixed nilpotent subgroup of $G$. We write $\inj(G,H)$ for the set of nilpotent injectors of $G$ containing $H$ and  $n_\mathfrak{I}(G,H) = \card{\inj(G, H)}$  for its cardinality.
Every nilpotent injector of $G$ contains the Fitting subgroup $\fitt(G)$ of $G$ and thus, using the notation above,  the total number of nilpotent injectors of $G$ is $n_\mathfrak{I}(G, \fitt(G))$, which  for simplicity we will write  as  $n_\mathfrak{I}(G)$. So   $n_\mathfrak{I}(G) = \card{\inj(G)}$.

Our first main result is the following.

\begin{prevtheorem} \label{Thm:A}
Let $G$ be an ${\mathcal N}$-constrained group. If $G$ is nilpotent, set $m_G=1$.  Otherwise, define
\begin{equation}
m_G \coloneqq \gcd\left\{ p - 1 : p\in\mathbb{P},\, p \mid (G : I) \,\,\text{ with }\,\, I \in \inj(G) \right\}.
\end{equation}
If $H \leq G$, then either $n_\mathfrak{I}(G,H)=0$ or $n_\mathfrak{I}(G,H) \equiv 1\md{m_G}$.
\end{prevtheorem}

Recall that the M\"obius function $\mu_P$ on the set of ordered pairs from a finite poset $P$ is defined recursively by
\[
\mu_P(x,y)= \left\{ \begin{array}{lll} 
0                            & \mbox{if } x \not \leq y,\\ 
1                            & \mbox{if } x = y,        \\ 
-\sum_{x \leq z<y}\mu_P(x,z) & \mbox{if } x < y .
\end{array} \right.
\]
Another goal of this paper is to show that divisibility properties also occur in the case of sums of M\"obius numbers defined on specific posets of subgroups of $G$.  In particular, we show the following.

\begin{prevtheorem}\label{Thm:B}
Given a nontrivial ${\mathcal N}$-constrained group $G$ with nilpotent injector $I$, 
let $\mathcal{X}$ be the set of subgroups of $G$ that are contained in at least one conjugate of $I$.  
Order $\mathcal{X}$ by inclusion and set
\[
\theta_\mathcal{X}(\fitt(G)) \coloneqq
\sum_{K \in \mathcal{X}}\mu_\mathcal{X}(\fitt(G),K).
\]
Then $\theta_\mathcal{X}(\fitt(G))$ is divisible by $(I : \fitt(G))$.
\end{prevtheorem}

There is a topological interpretation of \autoref{Thm:B}, 
as there is whenever a M\"obius function is involved. 
Given any poset $P$, 
the order complex $\Delta P$ is the abstract simplicial complex 
whose $k$-dimensional faces are the totally ordered subsets of size $k+1$ from $P$.  
Let $\hat{P}$ be the poset obtained from $P$ 
by adding a minimum element $\hat{0}$ and a maximum element $\hat{1}$.  
The reduced Euler characteristic $\tilde{\chi}(\Delta P)$ satisfies
\[
\tilde{\chi}(\Delta P) = \mu_{\hat{P}}(\hat{0}, \hat{1}),
\]
as was proved by Philip Hall and appears as \cite[Thm.~3.8.5]{StanleyEC1}.  
Let $\mathcal{X}_{>\fitt(G)}$ be the poset of all proper subgroups of $G$ 
strictly containing $\fitt(G)$ and contained in at least one nilpotent injector for $G$.  
\autoref{Thm:B} says that $\tilde{\chi}(\Delta\mathcal{X}_{>\fitt(G)})$ is divisible by $(I : \fitt(G))$.  There are various results about divisors of reduced Euler characteristics of posets of subgroups in the literature, some of which appear in \cite[Sections 5.3 and 5.4]{Smith}.

We write $\NN= \NN(G)$ for the poset of nilpotent subgroups of $G$. Then in view  of~\autoref{Thm:B}, one might ask whether $|I|$ divides 
$$
\NS:=\sum_{N \in \NN(G)}\mu(1,N).
$$  
Certainly such divisibility occurs when $G$ is nilpotent.  We will see, however, that $|I|/\NS$ can be an integer of arbitrarily large absolute value.  Indeed, we have the following result.

\begin{prevtheorem}\label{Thm:injectorratio}
If $p$ is a prime and $a$ is a positive integer, 
then there exist a soluble group $G$ 
and a nilpotent injector $I \leq G$ 
such that
\[
|I|/\NS=-p^a.
\]
\end{prevtheorem}

The theorem of Iranzo, Medina, and P\'{e}rez Monasor stated at the beginning (\cite[Theorem 0.1]{arithmetical_questions}) holds unconditionally for $\pi$-separable groups, but an analogous result, saying that the number $n_C(G,X)$ of Carter subgroups of a soluble group $G$ that contain a given nilpotent subgroup $X \leq G$ divides the number $\card{\cart(G)}$ of Carter subgroups, does not hold in general (see \autoref{Thm:CountrexampleDivisibilityCart}).  
However, if $G$ has a normal Hall $q'$-subgroup (for some prime divisor $q$ of $|G|$) then one can obtain a simple formula for $n_C(G,X)$, see \autoref{Thm:EOnePrime}.  It follows that the desired analogous result holds when $G$ has a Sylow tower, see \autoref{Cor:SylowTower}.

The organization of the paper is as follows. In Section~\ref{Sec:2} we first present the necessary background regarding nilpotent injectors and then we give the proofs of \autoref{Thm:A}, \autoref{Thm:B} and \autoref{Thm:injectorratio}. In Section~\ref{Sec:3} we consider Carter subgroups and we prove \autoref{Thm:CountrexampleDivisibilityCart} and \autoref{Thm:EOnePrime}.

All groups in this paper are assumed to be finite and we close this introduction by outlining some convenient notation.  For $n \in \mathbb{N}$, we will write $\varpi(n)$
for the set of prime numbers that divide $n$
and set $\varpi(G) \coloneqq \varpi(|G|)$ when $G$ is a group.
Unless we explicitly say otherwise, 
$X_p$ will stand for a Sylow $p$-subgroup of $X$. 
We will mainly use this notation when $X$ is nilpotent 
in which case $X_p$ is unique. Finally, we denote by $\mathbb{P}$ the set of all prime numbers.

\section{Nilpotent Injectors and M\"{o}bius Function}\label{Sec:2}

For a soluble group $G$, 
the authors in~\cite[Cor.~1.3]{arithmetical_questions} prove
that for $X$, $K$ subgroups of $G$
containing the Fitting subgroup $\fitt(G)$, with $X \leq K$ and $X$ nilpotent,
$n_\mathfrak{I}(K, X)$ divides $n_{\mathfrak{I}}(G)$
and in particular, if $\fitt(G) \leq K \leq G$ then
$n_{\mathfrak{I}}(K)$ divides $n_{\mathfrak{I}}(G)$.
They remark that both the assumption of solubility of $G$
and of the inclusion $\fitt(G) \leq K$ are necessary for
the conclusion to remain valid.

At the cost of being less precise,
\autoref{Thm:A} removes those two restrictions.
Before presenting the proof of~\autoref{Thm:A}, some preliminary remarks are in order.

Firstly, the modulus we are working with enjoys two key properties:
\begin{enumerate}[label={\upshape(\roman*)}]
\item\label{1} For each divisor $d$ of $(G : \fitt(G))$
we have $d \equiv 1 \md{m_G}$. 
In fact, $m_G$ is the greatest positive integer with that property.
\item\label{2} If $X \geq \fitt(G)$, then $m_G \mid m_X$.
\end{enumerate}

Secondly, nilpotent subgroups that contain the Fitting subgroup 
behave quite similarly to Sylow subgroups.
In particular, Theorem~2 in~\cite{mann_injectors} asserts the following:

Let $G$ be $\mathcal{N}$-constrained. 
Let $F = \fitt(G)$, $S$ be a nilpotent injector of $G$ 
and $H$ be a subgroup containing $F$.
\begin{enumerate}[label={\upshape(\alph*)}]
\item\label{alpha} $H$ is $\mathcal{N}$-constrained. 
Denote by $T$ a nilpotent injector of $H$.
\item\label{beta} $T = S_1 \cap H$ for some nilpotent injector $S_1$ of $G$.
\item\label{gamma} $S \cap H$ is contained in some nilpotent injector of $H$.
\item\label{delta} If $H \leq K$, then $T$ is contained in a nilpotent injector of $K$.
\item\label{epsilon} If $S \leq H$, then $S$ is a nilpotent injector of $H$.
\end{enumerate}

 The following remark is proved by Mann when he shows that injectors exist while proving \cite[Theorem 1]{mann_injectors}.

\begin{remark}[Mann]
\label{Rem:Mann}
Let $G$ be $\mathcal{N}$-constrained,  $F = \fitt(G)$ and  $F_{p'} = \oh^p(\fitt(G))$. If  $\C_p \coloneqq \cent_G(F_{p'})$ and $S_q \in \syl_q(\mathbf{C}(q))$ for $q , p $ distinct primes in $\varpi (G)$, then $\C_p$ is a normal subgroup of $G$ that also   normalises  $S_q$.
Furthermore,   $S_p$ is  the unique Sylow $p$-subgroup
of a nilpotent injector of $G$ and $S_p$ centralises $S_q$ for all  $p,q$ as above.  Finally,  for any  nilpotent subgroup $K$ of $G$ containing $F$, its
 Sylow $p$-subgroup $K_p$ 
centralises $F_{p'}$ and thus $K_p \leq \C_p$.     
\end{remark}

Before giving the proof of \autoref{Thm:A} we will need one extra piece of information regarding nilpotent  injectors that was already noted in \cite{arithmetical_questions}.

\begin{lemma}[Iranzo et al.]
\label{Lem:Product}
Let $G$ be an $\mathcal{N}$-constrained group, $I$ a nilpotent injector of $G$ and  $H \leq I$ a fixed nilpotent subgroup of $G$. 
For each $p \in \varpi(\fitt(G))$, 
let $H_p$ be the unique Sylow $p$-subgroup of $H$
and note that $H_p \leq I_p \leq \C_p = \cent_G(F_{p'})$.
Then
\[
n_\mathfrak{I}(G, H) = \prod_{p \in \varpi} n_p(\C_p, H_p),
\]
where 
for a fixed $p$-subgroup $X$ of $Y$, 
$n_p(Y, X)$ is the number of Sylow $p$-subgroups of $Y$ that contain $X$.
\end{lemma}
\begin{proofof}
This follows from the first formula in the proof of Corollary~1.3 in~\cite{arithmetical_questions}.  
\end{proofof}

We can now prove our first Theorem.

\begin{proofofA}
Suppose that the claim is false and let $G$ be a counterexample.
Among subgroups for which $G$ fails to satisfy the conclusion of the theorem,
let $H$ have maximum possible order, 
so that $n_\mathfrak{I}(G, H) \neq 0$ 
and $n_\mathfrak{I}(G, H) \not\equiv 1\md{m_G}$.
In particular, $H$ is necessarily nilpotent and we may assume that $H \leq I$ 
for a fixed nilpotent injector $I$ of $G$. 
We argue that $n_\mathfrak{I}(G, H) = n_\mathfrak{I}(G, H \fitt(G))$
since every nilpotent injector of $G$ contains $\fitt(G)$.
Thus, by the extremal choice of $H$, we see that $H \geq \fitt(G)$.

Let $\varpi = \varpi(G)$ be the set of primes that divide $\card{G}$ and write 
$\varpi= \pi_1 \cup \pi_2 \cup \alpha$ with 
\begin{align*}
\pi_1  &= \{ p \in \varpi  : F_p \in \syl_p(G) \}\\
\pi_2  &= \{ p \in \varpi : I_p \in \syl_p(G), I_p > F_p \}\\
\alpha &= \{ p \in \varpi : p \mid (G:I)\}.  
\end{align*}
Observe that the above sets may be empty, 
but if $p \in \pi_1$ then $p \nmid (G : F)$  
and if $q \in \pi_2$ then $q \nmid (G : I)$ and $q \mid (G : F)$. 
Finally, $m_G =  \gcd\left\{ p - 1 : p \in \alpha \right\}$. 
		
Now let $\varpi(\fitt(G))$ be the set of primes that divide $\card{\fitt(G)}$ 
and note that 
\[
\varpi(\fitt(G))= \pi_1 \cup \pi_2 \cup \pi_3
\]
with $\pi_3 \subseteq \alpha$ 
(by the preliminary remarks we have $\varpi(\fitt(G)) = \varpi(I)$). 
So 
\[
\fitt(G) = F_{\pi_1} \times F_{\pi_2} \times F_{\pi_3}.
\]
		
For each $p \in \varpi(\fitt(G))$, 
let $H_p$ be the unique Sylow $p$-subgroup of $H$. Then according to \autoref{Lem:Product} we have  $H_p \leq I_p \leq \C_p = \cent_G(F_{p'})$ and 
\[
n_\mathfrak{I}(G, H) = \prod_{p \in \varpi} n_p(\C_p, H_p),
\]
where  $n_p(\C_p, H_p)$ is the number of Sylow $p$-subgroups of $\C_p$ that contain $H_p$.
		
Assume  first that  $p \in \pi_1$. 
Then $n_p(\C_p, H_p)= 1$ 
since $\C_p$ contains the unique Sylow $p$-subgroup of $G$
(whose uniqueness persists in $\C_p$).
		
Now for each $p \in \pi_2 \cup \pi_3$, 
let $\beta = (\pi_1 \cup \pi_2 ) \setminus \{ p\}$ 
and $I_{\beta}= K$ be the Hall $\beta$-subgroup of $I$. 
Then 
\[
K = \prod_{q \in \pi_1} F_q \times \prod_{q \in \pi_2 \setminus \{p\} } I_q
\]
and $K$ is moreover a Hall $\beta$-subgroup of $G$.
		
Let  $Q = \C_p \cap K$ and observe that 
$Q$ is a Hall $\beta$-subgroup of $\C_p$ as $\C_p \unlhd G$. 
By the construction of $\C_p$ (see \autoref{Rem:Mann}) 
we get that  $\C_p$ normalises $I_q$ for all $q \neq p$ and thus $\C_p$ normalises $Q$.	
So the Schur--Zassenhaus theorem ensures 
that there exists a subgroup $T$ 
such that $\C_p = Q \rtimes T$ with $\varpi(T) \subseteq \alpha \cup \{p\}$. 
Without loss 
(as all such $T$ form a $\C_p$-conjugacy class) 
we may assume that $H_p \leq I_p \leq T$. 
Now,  $I_p$ is a Sylow $p$-subgroup of $\C_p$ that centralises $Q$. 
Hence 
\[
n_p(\C_p, H_p)   = \card{\{ I_p^g : H_p \leq I_p^g,\, g \in \C_p \}} 
                = \card{\{ I_p^t : H_p \leq I_p^t,\, t \in T \}}
                = n_p(T, H_p). 
\]
But 
\[
n_p(T, H_p) \equiv 1 \md{m_p}
\]
by Corollary~8 in~\cite{congruences},
where 
\[
m_p = \gcd\{q-1 : q \mid \card{T},\, q \neq p\}.
\]
As $\varpi(T) \setminus \{p\} \subseteq \alpha $, 
we get  $m_G \mid m_p$. 
Hence  $n_\mathfrak{I}(\C_p, H_p ) \equiv 1\md{m_G}$ 
for all $p \in \pi_2 \cup  \pi_3$ and so  $n_\mathfrak{I}(G, H) \equiv 1\md{m_G}$, 
against our initial assumption.
This contradiction shows that the original claim is valid and thus our proof is complete.	
\end{proofofA}

We wish to present applications of \autoref{Thm:A} 
that are not covered by Corollary~1.3 in~\cite{arithmetical_questions}
or Corollary~8 in~\cite{congruences}.

\begin{example}
Let $R$ be an $\mathcal{N}$-constrained insoluble group of even order such that $\fitt(R) = \mathbf{O}_2(R)$, 
let $S$ be a nontrivial group of odd order, and set $G = R \times S$.
Since $S$ is itself $\mathcal{N}$-constrained, being soluble,
we see that $G$ is $\mathcal{N}$-constrained.
Now observe that $\fitt(G) = \fitt(R) \times \fitt(S)$
and that the nilpotent injectors of $R$ are precisely its Sylow $2$-subgroups.
It follows that a nilpotent injector $I$ of $G$ 
(which exists owing to the fact that $G$ is $\mathcal{N}$-constrained)
contains a full Sylow $2$-subgroup of $R$ and thus has odd index in $G$.
Since $p-1$ is even for all primes $p$ such that $p \mid (G : I)$,
we see that $2 \mid m_G$ and so we may conclude that 
for all subgroups $H$ of $I$, $n_\mathfrak{I}(G, H)$ is odd.
An example of a group $R$ with the stated properties 
is $\mathrm{AGL}_n(2)$ with $n > 2$.
\end{example}

Given a group $G$ and a poset $\mathcal{X}$ of subgroups of $G$,
we define the function $\theta_{\mathcal{X}}^G : \mathcal{X} \rightarrow \mathbb{Z}$ by
\[
\theta_{\mathcal{X}}^G(H) = \sum_{K \in \mathcal{X}} \mu_{\mathcal{X}}(H, K) 
                        = \sum_{H \leq K \in \mathcal{X}} \mu_{\mathcal{X}}(H, K)\,.
\]
In the case $G$ is clearly understood, we write $\theta_{\mathcal{X}}^G(H) = \theta_{\mathcal{X}}(H) $. 
A consequence of a result of Brown~\cite{pfractional} is that 
if $\mathcal{X}$ is the poset of all $p$-subgroups of $G$,
where $p$ is some fixed prime divisor of $\card{G}$,
then $\card{G}_p$ divides $\theta_{\mathcal{X}}(1)$,
where $\card{G}_p$ is the order of a Sylow $p$-subgroup of $G$.
In fact, Brown's theorem was subsequently further generalised by Hawkes, Isaacs, and \"Ozaydin in~\cite[Thm.~5.1]{mobius}.

We wish to provide an analogue of Brown's theorem 
for subgroups contained in at least one nilpotent injector of $G$.
So, let 
\[
\mathcal{X} = \{H: H \leq I\,\, \mbox{ for at least one }\, I \in \inj(G)\},
\]
where as before $I$ is some fixed nilpotent injector of $G$
(the implicit assumption here is that $G$ is non-trivial and $\mathcal{N}$-constrained).
Then $\theta_{\mathcal{X}}(1) = 0$ by Lemma~2.4 in~\cite{mobius} 
since $\mathrm{Core}_G(I) = \fitt(G) > 1$.
However, the situation becomes more interesting 
by restricting to the \enquote{top part} of $\mathcal{X}$, 
i.e. by focussing on
\[
\mathcal{Y}= \{H \in \mathcal{X} : \fitt(G) \leq H\}\,. 
\]
Observe that $\theta_{\mathcal{X}}(\fitt(G))= \theta_{\mathcal{Y}}(\fitt(G))$ because 
$\mu_{\mathcal{X}}(\fitt(G), K) = 0 $ for all $K \notin \mathcal{Y}$. 
We are now ready to give the proof of  \autoref{Thm:B}.

\begin{proofofB}
As before, let $\varpi = \varpi(\fitt(G))$ 
and write $F$ in place of $\fitt(G)$.
Let $\C_p = \cent_G(F_{p'})$ and $K$ any nilpotent subgroup of $G$ containing $F$. In view of \autoref{Rem:Mann} we get $K_p \leq \C_p$ and thus   the poset of all $p$-subgroups of $\C_p$ containing $F_p$ 
coincides with the poset of all $p$-subgroups of $\mathcal{X}$ containing $F_p $. 

We denote the above poset by $\mathcal{Y}_p$, that is 
\[
\mathcal{Y}_p = \{ U \in \mathcal{X} : F_p \leq U \text{ is a $p$-group} \} = \{ F_p \leq U \leq \C_p : U \text{ is a $p$-group}\}.  
\]
Now, if $F=\fitt(G) \leq K \in \mathcal{X}$ 
then $K = \prod_{p \in \varpi} K_p$ 
with $K_p \in \mathcal{Y}_p$ and thus 
\[
\mu_{\mathcal{X}}(F, K) = \prod_{p \in \varpi} \mu_{\mathcal{X}}(F_p, K_p)\,,
\]
by Theorem~8.4 in~\cite{mobius}.
Next, notice that any selection $\{ K_p \in \mathcal{Y}_p : p \in \varpi \}$ 
defines a nilpotent subgroup $K = \prod_{p \in \varpi} K_p$ of $G$ above $F$. 
Hence $K$ is contained in a nilpotent injector of $G$ and so $K \in \mathcal{X}$. 
We conclude that 
\begin{equation}
\theta_{\mathcal{X}}(F) =  \sum_{F \leq K \in \mathcal{X}} \mu_{\mathcal{X}}(F, K) 
                        = \sum_{F \leq K \in \mathcal{X}}
                           \prod_{p \in \varpi} \mu_{\mathcal{X}}(F_p, K_p) 
                        = \prod_{p \in \varpi} \sum_{U \in \mathcal{Y}_p} \mu_{\mathcal{X}}(F_p, U)\,.
\end{equation}
Applying Corollary~3.9 in~\cite{mobius} 
with $\C_p$ in place of $G$
and $F_p$ in place of $H$ yields that 
\[
(\C_p : F_p)_p = (I : F)_p \quad\quad \mbox{divides} \quad \sum_{U \in \mathcal{Y}_p} \mu_{\mathcal{X}}(F_p, U).
\]
It follows that
$\, (I : F) = \prod_{p \in \varpi} (I : F)_p\, $  divides 
$\, \prod_{p \in \varpi} \sum_{U \in \mathcal{Y}_p} \mu_{\mathcal{X}}(F_p, U)
= \theta_{\mathcal{X}}(F)\, , $
completing the proof.
\end{proofofB}

We mention that \autoref{Thm:B} is not true if $I$
is an arbitrary nilpotent subgroup of $G$ 
(with $\fitt(G)$ being naturally replaced by $\mathrm{Core}_G(I)$)
nor is it true if $I$ is specifically the dual of a nilpotent injector,
i.e. a Carter subgroup of $G$.
Concretely, the group \texttt{SmallGroup(54,5)} has a core-free
Carter subgroup $C$ that is cyclic of order $6$. 
But for $\mathcal{X} = \{1 \leq H \leq C^g : g \in G\}$ we get 
$(C : \mathrm{Core}_G(C)) = 6$ and  $\theta_{\mathcal{X}}(1) = -2$. 
Moreover, in the context of~\autoref{Thm:B} it is \textbf{not} true
that $\card{I}$ divides $\theta_{\mathcal{X}}(1)$.
This fails to be the case in $G = S_3 \times S_4$ wherein $I \cong C_3 \times D_8$,
but $\theta_{\mathcal{X}}(1) = -12$.

We now present the proof of \autoref{Thm:injectorratio}.

\begin{proofofC}
Given a prime $p$ and prime power $q = p^n$, 
we consider the affine general semi-linear group $A \coloneqq \mathrm{A\Gamma L}_1(q)$, 
whose elements are triples $(\tau, a, b)$ 
with $\tau \in \aut(\ff_q)$, $a \in \ff_q^\ast$, and $b \in \ff_q$. 
The group structure is determined by the action on the set $\ff_q$ 
in which $(\tau, a, b)$ maps $x$ to $a\tau(x) + b$. 
Writing $\iota$ for the identity map on $\ff_q$, 
we consider three subgroups of $A$:
\begin{itemize}
\item $J \coloneqq \{(\tau, 1, 0) : \tau \in \aut(\ff_q)\}$;
\item $M \coloneqq \{(\iota, a, 0) :a \in \ff_q^\ast\}$; and
\item $V \coloneqq \{(\iota, 1, b) : b \in \ff_q\}$.
\end{itemize}
We observe that $J$ is cyclic of order $n$, 
$M$ is cyclic of order $p^n-1$, 
and $V$ is elementary abelian of order $p^n$.  
It is straightforward to confirm the conjugation rules
\begin{equation}\label{con1}
(\iota, 1, b)^{(\iota, a, 0)} = (\iota, 1, ab)
\end{equation}
and
\begin{equation}\label{con2}
(\iota, a, b)^{(\tau, 1, 0)} = (\iota, \tau(a), \tau(b)).
\end{equation}
It follows that $M$ normalises $V$ and $J$ normalises $MV$, 
hence $A$ is soluble.

Now we restrict ourselves to the case where $n$ is a power of $p$, $n = p^a$, 
and consider a subgroup of $A$. 
Let $r$ be a Zsigmondy prime for the pair $(p, n)$. 
So, $r$ divides $p^n-1$ but does not divide $p^j-1$ for $1 \leq j < n$. 
Such an $r$ exists by Zsigmondy's Theorem, see~\cite{Zsigmondy}. 
Let $R \leq M$ have order $r$. 
As $R$ is characteristic in $M$, $J$ normalises $R$ and so normalises $RV$. 
We form the group $G \coloneqq JRV \leq A$ and observe that $|G| = p^{a+n}r$.

By~\eqref{con1}, $\cent_V(R) = 1$. 
Since $\N_V(R) = \cent_V(R)$, it follows that $\N_G(R) = JR$.
We claim that $\cent_J(R) = 1$, 
from which it follows that $\cent_G(R) = R$. 
For each $\tau \in J$, the set $\{x \in \ff_q : \tau(x) = x\}$ is a subfield of $\ff_q$, which is proper unless $\tau = 1$. 
Since $r$ is a Zsigmondy prime for $(p, n)$, no proper subfield of $\ff_q$ contains an element of multiplicative order $r$. 
The claim now follows from~\eqref{con2}.

We see now that every nilpotent subgroup of $G$ either has order $r$ or is a $p$-subgroup.  
Let $\SS^+_p(G)$ be the poset of $p$-subgroups of $G$ (including $1$).  
Since $\oh_p(G) = V \neq 1$, we see that
\[
\sum_{P \in \SS^+_p(G)}\mu(1,P) = 0
\]
(see~\cite[Prop.~3.8.6]{StanleyEC1} and~\cite[Prop.~2.4]{Quillen}).  
As $\mu(1,R) = -1$, we conclude that
\begin{equation}\label{nilsum}
\NS = - \card{\syl_r(G)} = - (G : \N_G(R)) = -p^n.
\end{equation}
We claim that $I \coloneqq JV$ is a nilpotent injector in $G$.  
The only maximal nilpotent subgroups of $G$ 
are the Sylow $p$-subgroups and the Sylow $r$-subgroups.  
A Sylow $r$-subgroup cannot be a nilpotent injector, 
since its intersection with the subnormal subgroup $V$ 
is not a maximal nilpotent subgroup of $V$.  
As $G$ is soluble and hence contains a nilpotent injector 
and $I$ is a Sylow $p$-subgroup, the claim follows.  
We conclude now that
\[
|I| / \NS = -p^{a+n}/p^n = -p^a,
\]
completing the proof.
\end{proofofC}

Say $p_1$ and $p_2$ are distinct primes and we construct $G_1$ and $G_2$ of respective orders $p_1^{a_1+n_1}r_1$ and $p_2^{a_2+n_2}r_2$ as in the proof of \autoref{Thm:injectorratio}.  
If $r_1 \neq r_2$ then $\NN(G_1 \times G_2) \cong \NN(G_1) \times \NN(G_2)$, hence 
\[
\theta_\NN^{G_1 \times G_2}(1)=\theta_\NN^{G_1}(1) \theta_\NN^{G_2}(1)=p_1^{n_1}p_2^{n_2}
\]
by~\cite[Prop. 3.8.2]{StanleyEC1} and~\eqref{nilsum}.  
One can show that if $I_1$ and $I_2$ are, respectively, nilpotent injectors in $G_1$ and $G_2$, then $I_1 \times I_2$ is a nilpotent injector $G_1 \times G_2$ having order $p_1^{a_1+n_1}p_2^{a_2+n_2}$.  
Now
\[
\card{I_1 \times I_2} / \theta_\NN^{G_1 \times G_2}(1) = p_1^{a_1}p_2^{a_2}.
\]
One can repeat this product construction 
as long as one is not forced to choose a Zsigmondy prime that has been chosen already, 
thereby producing integers $k = p_1^{a_1}\ldots p_t^{a_t}$, with $t$ arbitrarily large, 
such that there is a soluble group $G_k$ with nilpotent injector $I_k$ satisfying
\[
\card{I_k} / \card{\theta_\NN^{G_k}(1)} = k.
\]
This leads to the following question.

\begin{question}
Is it true that for every positive integer $k$ 
there is a soluble group $G$ with nilpotent injector $I$ satisfying
\[
\card{I} / \card{\NS} = k?
\]
\end{question}

Suppose that $\mathcal{X}$ is a finite poset and write $\mathrm{Max}(\mathcal{X})$ for the set of maximal elements of $\mathcal{X}$. 
If $x \in \mathcal{X}$ then write $\nu_{\mathcal{X}}(x)$ 
for the number of maximal elements of $\mathcal{X}$ lying above $x$, so that
\[
\nu_{\mathcal{X}}(x) = \card{\{ z \in \mathrm{Max}(\mathcal{X}) : x \leq z\}}.
\]
Now if 
\[
\delta_{\mathcal{X}}(y) = \left\{\begin{array}{l}
1, \text { if } y \in \mathrm{Max}(\mathcal{X})\\
0, \text { otherwise }
\end{array}\right.
\]
then  
\[
\nu_{\mathcal{X}}(x) = \sum_{x \leq y  \in \mathcal{X}} \delta_{\mathcal{X}}(y).
\]
Hence, by M\"{o}bius inversion we get
\[
\delta_{\mathcal{X}}(x)= \sum_{x \leq y \in \mathcal{X}}\mu_{\mathcal{X}}(x, y) \nu_{\mathcal{X}}(y).
\]
Thus if $x \notin \mathrm{Max}(\mathcal{X})$ then 
\[
\sum_{x \leq y \in \mathcal{X}}\mu_{\mathcal{X}}(x, y) \nu_{\mathcal{X}}(y) = 0\,.
\]
So if $\nu_{\mathcal{X}}(y) \equiv 1 \md{m}$ for some modulus $m$ and for all $y \in \mathcal{X}$, 
then the assumption $x \notin \mathrm{Max}(\mathcal{X})$ implies that
\begin{equation}\label{Eq:mdividestheta}
\sum_{x \leq y \in \mathcal{X} }\mu_{\mathcal{X}}(x, y) \equiv 0 \md{m}.    
\end{equation}

Below we deduce some consequences  of~\eqref{Eq:mdividestheta}.

\begin{corollary}\label{Cor:First}
Let $G$ be a group and write $\mathcal{X}$ for the poset of all $p$-subgroups of $G$
(together with the trivial subgroup).
Let $\card{G}_p$ be the $p$-part of $|G|$ 
and 
\[
m_p = \gcd\{q - 1 : q \mid \card{G},\,\,q \in \mathbb{P}\setminus \{p\}\}.
\]
Then $\mathrm{lcm}\{\card{G}_p, m_p\}$ divides $\theta_{\mathcal{X}}(1)$.
\end{corollary}

\begin{proofof}
That $\card{G}_p$ divides $\theta_{\mathcal{X}}(1)$ is a consequence of Brown's theorem
and $m_p$ divides $\theta_{\mathcal{X}}(1)$ 
by Corollary~8 in~\cite{congruences} and~\eqref{Eq:mdividestheta} with $m = m_p$
and $x$ the trivial subgroup of $G$.
\end{proofof}

\begin{corollary}\label{Cor:Second}
Let $G$ be a $\pi$-separable group and  $\mathcal{X}$  the set of all $\pi$-subgroups of G. If  $K$  is a Hall $\pi$-subgroup of $G$ let $n = (G : \N_G(K))$ and $m = \gcd\{q - 1 : q \mid n,\,\, q \in \mathbb{P}\}$.
Then $\mathrm{lcm}\{|K|, m\}$ divides $\theta_{\mathcal{X}}(1)$.
\end{corollary}

\begin{proofof}
Certainly $\card{K}$ divides $\theta_{\mathcal{X}}(1)$
by Theorem 5.1 in \cite{mobius} and the fact that the development
property holds in $\pi$-separable groups.
Thus it suffices to justify that $m$ divides $\theta_{\mathcal{X}}(1)$ as well.
Now observe that for a fixed $\pi$-subgroup $X$ of $G$,
Theorem~0.1 in~\cite{arithmetical_questions} implies (taking $K = G$ in that theorem) 
that the number $n_{\pi}(G, X)$ of Hall $\pi$-subgroups of $G$ that contain $X$ divides $n$.
Then $n_{\pi}(G, X) \equiv 1 \md{m}$ 
since this congruence is satisfied by every divisor of $n$.
Thus $m$ divides $\theta_{\mathcal{X}}(1)$ by~\eqref{Eq:mdividestheta}, 
completing the proof.
\end{proofof}

\begin{corollary}\label{Cor:Third}
Let $G$ be a non-trivial $\mathcal{N}$-constrained group 
and write
\[
\mathcal{X} = \{1 \leq H \leq I^g : g \in G\},
\]
where $I$ is a nilpotent injector of $G$.
Let $n = (I : \fitt(G))$ and $m = m_G$,
where $m_G$ is as in \autoref{Thm:A}.
If $n > 1$, then $\mathrm{lcm}\{m,n\}$ divides $\theta_{\mathcal{X}}(\fitt(G))$.
\end{corollary}

\begin{proofof}
Since $n$ divides $\theta_{\mathcal{X}}(\fitt(G))$ by \autoref{Thm:B},
it will suffice to show that $m$ divides $\theta_{\mathcal{X}}(\fitt(G))$.
But this is true for the specified poset $\mathcal{X}$ by virtue of~\eqref{Eq:mdividestheta}
and \autoref{Thm:A} with $m = m_G$.
\end{proofof}

The requirement $n > 1$ is essential as the Frobenius group $G$ of order $21$ 
has $m_G = 2$ and $\theta_{\mathcal{X}}(\fitt(G)) = 1$.

Assume now that $\FF$ is a Fitting class
i.e.,  $\FF$ is  a collection of groups that includes $1$ and is closed under isomorphisms, under normal subgroups, and further, under the operation of forming normal products. 
Let $G$ be  a soluble group and let $\inj_{\FF}(G)$ be the set of $\FF$-injectors of $G$. 
Since $G$ is soluble, $\inj_{\FF}(G)$ is not empty and forms a single conjugacy class of $G$. 
If $G_{\FF}$ is the $\FF$-radical of $G$, 
i.e. the product of all normal $\FF$-subgroups of $G$, 
then the following holds.

\begin{proposition}
If $\FF$ is a Fitting class  and  $G$ a soluble group then   
$\card{G_{\FF}}$ divides $\sum_{H \in \FF} \mu(1, H)$.
\end{proposition}

\begin{proofof}
The $\FF$-radical $G_{\FF}$ acts by conjugation on the poset $\mathcal{P}$ of all $\FF$-subgroups of $G$. 
If $g \in G_{\FF}$ with $g \neq 1$, let 
\[
\mathcal{P}^g = \{K \in \mathcal{P} : K^g = K\}. 
\]
Note that if  $K \in \mathcal{P}^g$ then the group 
$H = \gensub{g}K$ is  the product of its normal $\FF$-subgroups $K$ and $H \cap G_{\FF}$  
and thus it is an $\FF$-subgroup  of  $G$ contained in $\mathcal{P}^g$. 
Hence any maximal element of $\mathcal{P}^g$ contains $\gensub{g}$. 
We conclude (see  Lemma~2.4 in~\cite{mobius}) that 
\[
\sum_{H \in \mathcal{P}^g} \mu(1,H) = 0,
\]
for every $1 \neq g \in G_{\FF}$. 
Now the proposition follows by Lemma~9.2 in~\cite{mobius}. 
\end{proofof}

As an immediate consequence of the proposition 
applied on the Fitting class $\NN = \NN(G)$ of all nilpotent subgroups of $G$ 
we get the following.

\begin{corollary}
If $G$  is a soluble group, 
then $\card{\fitt(G)}$ divides $\sum_{H \in \NN} \mu(1,H)$.
\end{corollary}

\section{Carter subgroups}\label{Sec:3}

If $G$ is a soluble group and $X \leq G$ we will write $\cart(G, X)$ 
for the set of Carter subgroups of $G$ containing $X$ and $n_C(G,X) = \card{\cart(G, X)}$ 
for its  cardinality.
We recall that  Carter subgroups exist in solvable groups and they are all conjugate. Furthermore,  if $C$ is a Carter subgroup of $G$ 
then $f(C)$ is a Carter subgroup of $f(G)$ 
for any group homomorphism $f$ of $G$. 
In particular, if $N$ is a normal subgroup of $G$ then $CN/N$ is Carter in $G/N$. 
Further, if $H/N$ is Carter in $G/N$ and $K$ is Carter in $H$, then $K$ is Carter in $G$. The interested reader could look at Section 9.5 in \cite{Robinson1982ACI} for the proofs of the  aforementioned results.  

As we mentioned in the Introduction,
with no extra assumptions 
there is no analogue for Carter subgroups of the theorem
of Iranzo, Medina, and P\'{e}rez Monasor
and we prove this claim in our next result.

\begin{prevtheorem}\label{Thm:CountrexampleDivisibilityCart}
If $p$ is an odd prime then there exist a soluble group $G$, 
a Carter subgroup $H \leq G$, and $g \in G$ 
such that the largest power of $p$ dividing $(G : H)$ is $p^{2^{p-1}}$ and the number of Carter subgroups of $G$ containing $\langle g \rangle$ is divisible by $p^{2^{p-1}+1}$.
\end{prevtheorem}

\begin{proofof}
Write $\ff$ for the field $\ff_{2^p}$ and, respectively, 
$\ff^+$ and $\ff^\ast$ for the additive and multiplicative groups of $\ff$.  
Set $A = \mathrm{A\Gamma L}_1(2^p)$.  
As in the proof of~\autoref{Thm:injectorratio}, we have $A = JMV$ with 
$J \cong \aut(\ff)$, $M \cong \ff^\ast$, and $V \cong \ff^+$.  
The group $A$ acts on the set $\ff$ as in the proof of~\autoref{Thm:injectorratio}.

Let $W$ be an $\ff_p$-vector space with basis $B \coloneqq \{e_x : x \in \ff\}$.  
The action of $A$ on $\ff$ determines a representation of $A$ on $W$ 
in which $A$ permutes $B$ as it permutes $\ff$.  
Using this representation, we form the semidirect product $G = AW$.  

Set $V_2 \coloneqq \cent_V(J)$.  
So, $|V_2| = 2$.  
Now set $H \coloneqq J V_2 \cent_W(V_2)$.  
Given $S \subseteq \ff$, set
\[
e_S \coloneqq \sum_{x \in S}e_x \in W.
\]
We will prove 
\begin{enumerate}[label={\upshape(\Alph*)}]
\item\label{item:A} $H$ is a Carter subgroup of $G$ having order $2p^{2^{p-1}+1}$, hence $(G : H) = 2^{p-1}(2^p-1)p^{2^{p-1}}$; and
\item\label{item:B} there exists $X \subseteq \ff$ such that 
the number of Carter subgroups of $G$ containing $\langle e_X \rangle$ 
is divisible by $p^{2^{p-1}+1}$,
\end{enumerate}
thereby proving the theorem.

First we prove~\ref{item:A}.  
Direct calculation shows that
\begin{equation}\label{Eq:gorder}
|G| = p(2^p-1)2^pp^{2^p}=p^{2^p+1}(2^p-1)2^p
\end{equation}
and 
\begin{equation}\label{Eq:horder}
|H| = |J||V_2||\cent_W(V_2)| = 2p|\cent_W(V_2)|.
\end{equation}
Now $V_2=\langle t \rangle$, where $t$ acts on $\ff$ by mapping $e_x$ to $e_{x+1_\ff}$ for each $x \in \ff$.  
(Here $1_\ff$ is the multiplicative identity.)  
Therefore,
\[
\cent_W(V_2) = 
\left\{\sum_{x \in \ff}\alpha_xe_x : \alpha_x \in \ff_p, \alpha_x = \alpha_{x+1}\,\, \mbox{ for all }\, x \right\}
\]
has order $p^{2^{p-1}}$.  
Direct calculation shows that $(G : H) = 2^{p-1}(2^p-1)p^{2^{p-1}}$, as claimed.  
We must show that $H$ is nilpotent and self-normalising in $G$.  
As $J$ centralises $V_2$ and normalises $W$, 
$J$ normalises $\cent_W(V_2)$.  
Therefore, $J \cent_W(V_2)$ is a Sylow $p$-subgroup of $H$ 
that centralises the Sylow $2$-subgroup $V_2$, hence $H$ is nilpotent.  

Now $\N_G(H) \leq \N_G(V_2) = \cent_G(V_2)$, 
and $\cent_G(V_2)/\cent_W(V_2)$ is isomorphic with a subgroup of $\cent_A(V_2) = JV$ (since $\cent_M(V_2) = 1)$.  
It follows that
\[
\cent_G(V_2) = JV\cent_W(V_2).
\]
Let $\tau$ generate $J$, with $\tau(x) = x^2$ for each $x \in \ff$.  
We will show that given $v \in V$, 
we have $\tau^v \in H$ if and only if $v \in V_2$, 
thereby completing the proof of~\ref{item:A}.  
We calculate 
(using in the last term addition in $\ff$ to denote multiplication in $V$ 
and $.$ to denote multiplication in $A$)
\[
\tau^v = v \tau v = \tau(v^\tau)v  =\tau.(v^2+v).
\]
It follows that $\tau^v \in H$ if and only if $v^2+v \in V_2$, 
that is, if and only if $v^2+v \in \{0_\ff,1_\ff\}$.  
If $v^2+v=0_\ff$ then $v \in V_2$.  
If $v^2+v=1_\ff$ then $v^3=1_\ff$, 
which is impossible, since $3$ does not divide $\card{\ff^\ast}$.

Now we prove~\ref{item:B} by showing that if $X$ is a subgroup of index $2$ in $\ff^+$ containing $1_\ff$ 
then the number of Carter subgroups of $G$ containing 
$\gensub{e_X}$ is divisible by $p^{2^{p-1}+1}$.  
Fix such an $X$.

Let $\cl_G(e_X)$ be the $G$-conjugacy class of $e_X$.  
Counting in two ways the pairs $(s,C)$ with $s \in \cl_G(e_X)$  
and $C$ a Carter subgroup of $G$ containing $s$, 
we see that the number $\nu(e_X)$ of Carter subgroups containing $e_X$ satisfies
\begin{equation}\label{pairs}
\nu(e_X) = \frac{\card{\cent_G(e_X)} \card{\cl_G(e_X) \cap H}}{|H|}.
\end{equation}

Given~\eqref{Eq:horder}, 
it suffices now to show that $\card{\cent_G(e_X)} \card{\cl_G(e_X) \cap H}$ is divisible by $p^{2^p+2}$.

We observe that $\cent_G(e_X) = \cent_A(e_X) W$ 
and that $\cent_A(e_X)$ is the setwise stabilizer of $X$ in the action of $A$ on $\ff$.  
Moreover, $a \in A$ maps $e_X$ to $e_{a(X)}$, 
thus the $G$-conjugates of $e_X$ are in bijection with the $A$-conjugates of $X$. 

If $a \in JM = \mathrm{\Gamma L_1(2^p)}$ then $a(X)$ is a subgroup of index $2$ in $\ff^+$, 
while if $v \in V$ then $v(X)$ is a coset of such a subgroup.  
Therefore, every $A$-conjugate of $X$ is a coset of a subgroup of index $2$ in $\ff^+$.  
In fact, every such coset is an $A$-conjugate of $X$.  
Indeed, given any such subgroup $K$, there is some $v \in V$ mapping $K$ to its non-trivial coset.  
Moreover, $M$ permutes the set of index $2$ subgroups regularly: 
if $K$ is such a subgroup and $\beta \in \ff^\ast$ with $\beta K = K$, 
then the multiplicative order of $\beta$ divides both $\card{\ff^\ast} = 2^p-1$ and $\card{K \setminus \{0_\ff\}} = 2^{p-1}-1$ 
(since $\beta$ acts semiregularly on $\ff \setminus \{0_\ff\}$).  
Therefore $\beta = 1$.  
Since the number of index $2$ subgroups of $\ff^+$ is $2^p-1 = |M|$, 
our claim follows.

We see now that $|\cl_G(e_X)| = 2(2^p-1)$.  
By~\eqref{Eq:gorder}
\[
\card{\cent_G(e_X)} = 2^{p-1} p^{2^p+1}.
\]
It remains to show that $\card{\cl_G(e_X) \cap H}$ is divisible by $p$.  
We know that the elements of $\cl_G(e_X)$ are those $e_S$ 
such that $S$ is a coset of an index $2$ subgroup of $\ff^+$.  
Since such an $e_S$ lies in $W$, 
we have $e_S \in H$ if and only if $e_S \in \cent_W(V_2)$, 
that is, if and only if $S + 1_\ff = S$.  
If $S \leq \ff^+$ then $S + 1 = S$ if and only if $1 \in S$.  
If $S$ is the non-trivial coset of $K \leq \ff^+$ then $S+1=S$ if and only if $K + 1 = K$.

We conclude that $\card{\cl_G(e_X) \cap H}$ 
is twice the number of subgroups of index $2$ in $\ff^+$ that contain $1_\ff$.  
By transitivity of $M$ on $\ff \setminus \{0_\ff\}$, 
every $\alpha \in \ff^\ast$ is contained in the same number of index $2$ subgroups of $\ff^+$.  
As there are $2^p-1$ such subgroups, 
each containing $2^{p-1}-1$ elements of $\ff^\ast$, 
we see that
\[
\card{\cl_G(e_X) \cap H} = 2(2^{p-1}-1),
\]
and Fermat's Little Theorem implies that $p$ divides $\card{\cl_G(e_X)}$. 
\end{proofof}

We now begin work to show that in the presence of a normal Hall subgroup
there is such a version of \enquote{divisibility}  
(cf. \autoref{Thm:EOnePrime} and the corollaries that follow) 
and start with the next lemma which is a generalisation of Lemma 1.1 in~\cite{arithmetical_questions}.

\begin{lemma}\label{Lem:CarterSemidirect}
Let $G$ be a finite group with $X \leq C \leq G$ and assume $N \unlhd G$ with $G = N \rtimes  C$. 
If $\mathcal{A} = \{ C^g : g \in G\}$ is the conjugacy class of $C$ in $G$, then 
\[
\card{\{ C^g \in \mathcal{A} : X \leq C^g\}} = (\cent_N(X) : \cent_N(C))\,. 
\]
\end{lemma}

\begin{proofof}
If $C^g \in \mathcal{A}$ with $X \leq C^g$ 
then  $C^g = C^n$ for some $n \in N$ and $X \leq C^n \cap C$. 
But $C \cap C^n \leq \cent_C(n)$.  
Hence  $X \leq \cent_C(n)$ and thus 
\[
\cent_N(X) = \{n \in N : X \leq C^n\}.
\] 
Therefore, $\cent_N(X)$ acts transitively on the set  $\{C^n \in \mathcal{A} : X \leq C^n\}$, 
which is the orbit of $C$ under this action. 
The stabiliser of $C$ is $\cent_{\cent_N(X)}(C) = \cent_N(C)$,
so 
\[
\card{\{C^n : X \leq C^n \}} = (\cent_N(X) : \cent_N(C))
\]
and the proof is complete. 
\end{proofof}

The above implies:

\begin{proposition}\label{Prop:SemidirWithCart}
Assume that $G$ is a finite soluble group, $X \leq C \leq G$ with $C \in \cart(G)$, 
and $N \unlhd G$ such that $G = N \rtimes C$. 
Then $n_C(G, X) = \card{\cent_N(X)}$  divides $\card{\cart(G)}$.
\end{proposition}

\begin{proofof}
As $C \in \cart(G)$ we have $\cent_N(C) \leq \N_G(C) \cap N  = C \cap N = 1$. 
Hence \autoref{Lem:CarterSemidirect} implies that $n_C(G, X) =  \card{\cart(G, X)}  = \card{\cent_N(X)}$, 
since the Carter subgroups of $G$ form a single $G$-conjugacy class. 
But $\card{\cart(G)} = (G : \N_G(C)) = (G : C) = |N| $, 
and the proposition follows.
\end{proofof}

The following lemma was communicated to the first author by I.M. Isaacs in the Winter of 2016 and with his kind permission we include it here.

\begin{lemma}[Isaacs]\label{Lem:normalhall-Isaacs}
Let $N$ be a normal Hall $\pi$-subgroup of the finite soluble group $G$,
where $\pi \subseteq \varpi(G)$, 
and let $Q \in \Hall_{\pi^{\prime}}(G)$. 
If $C$ is a Carter subgroup of $G$ then $C_{\pi'}$ is a  Carter subgroup of some conjugate of $Q$. 
Furthermore, for a given Carter subgroup $D$ of the fixed subgroup $Q$ 
we have 
\[
\cart(G, D) = \{ DU : U \in \cart(\cent_N(D))\}.
\]

\end{lemma}
\begin{proofof}
Let $C \in \cart (G)$ and $D \in \cart(Q)$.  Both $CN$ and $DN$ are full preimages of Carter subgroups of $G/N$ and so, by conjugating appropriately, we may assume that $CN=DN$. Then $D$ is a Hall  $\pi'$-subgroup of $DN$,  and $C_{\pi'}$ is a Hall  $\pi'$-subgroup of $CN$, given that $N$ is a $\pi$-group. 
We conclude that $C_{\pi'}= D^x$ for some $x \in G$. As  $D^x$ is a Carter subgroup of $Q^x$, 
the first claim of the lemma follows.  

To prove that $\cart(G, D)=   \{ DU : U \in \cart(\cent_N(D))\}$, 
we first fix   $C \in \cart (G, D)$ and we claim that $C \in  \{ DU : U \in \cart(\cent_N(D))\}$. Indeed, $C = U \times D$ with $U \coloneqq C \cap N$ being the unique Hall $\pi$-subgroup of $C$. Note that $U$ centralises $D$, thus $U \leq Y \coloneqq \cent_N(D)$, and we want that $U$ be Carter in $Y$. But $U$ is nilpotent, so it suffices that $\N_Y(U) = U$. If some $g \in Y$ normalises $U$, then it normalises $C$, since it centralises $D$, thus $g \in C \cap Y \leq U$.

Conversely, let $C \coloneqq DU$ where  $D \in \cart(Q)$ and  $U \in \cart(\cent_N(D))$  as before. The aim is to show that $C \in \cart(G, D)$.   Clearly, both $U$ and $D$ are nilpotent mutually centralising subgroups of $C$, thus both normal in $C$, and it follows that $C$ is nilpotent. To complete the proof, it thus suffices to establish that $\N_G(C) = C$. 
First, let us prove that $C$ is Carter in $X \coloneqq ND$. 

By Dedekind, $\N_X(C) = \N_N(C)D$, thus it suffices that $\N_N(C) = U$. 
The uniqueness of Hall subgroups in nilpotent groups 
forces $\N_N(C)$ to normalise both $U$ and $D$ and thus $\N_N(C) = \N_N(U ) \cap \N_N(D)$. 
But  $\N_N(D) = \cent_N(D)$
and so
\[
\N_N(C) =  \N_N(U ) \cap \cent_N(D)= \N_{\cent_N(D)}(U). 
\]
It follows that $\N_N(C)= U$,  since  $U \in \cart(\cent_N(D))$. 

Now, $X/N$ is Carter in $G/N$, and $C$ is Carter in $X$. 
Therefore, $C$ is Carter in $G$ by our preliminary remarks at the beginning of the section, 
and the proof is complete.
\end{proofof}

In the case that $N$ is a Hall $q'$-subgroup of $G$, Isaacs's lemma  simplifies to: 
\begin{corollary}\label{Cor:IsaacsOnePrime}
Let $N$ be a normal Hall $q'$-subgroup of the finite soluble group $G$,
where $q $ is a prime dividing the order of $G$  
and let $Q \in \syl_{q}(G)$. 
Then every Carter subgroup $C$ of $G$ contains a conjugate of $Q$ and thus $C_q  \in \syl_q(G)$.
Furthermore, for a  fixed  Sylow $q$-subgroup $Q$  of $G$ 
we have 
\[
\cart(G, Q) = \{ QU : U \in \cart(\cent_N(Q))\}.
\]

\end{corollary}
\begin{proof}
The only Carter subgroup of a $q$-group $Q$  is $Q$ itself. 
\end{proof}

We can now prove 

\begin{prevtheorem}\label{Thm:EOnePrime}
Let $G$ be soluble, and suppose that $N$
	is a normal Hall $q'$-subgroup of $G$, where $q \in \varpi(G)$. 
	Let $X \leq G$ and assume that
	$X$ is contained in some Carter subgroup $C$ of $G$. 
	Write $X = X_{q'}  \times X_{q}$ and $C = C_{q'} \times C_{q}$ 
	for their respective decompositions according to Sylow $q$- and Hall $q^{\prime}$-subgroups.
 Let $A = \cent_G(X_{q'})$  and note that  $X_{q} \leq C_{q} \leq A$. Then $C_q$ is a Sylow $q$-subgroup of both  $G$ and  $A$ and  
 \begin{equation}
 n_C(G,X) \, =  |\syl_q(A, X_q)| \cdot  n_C(\cent_N(C_{q}), X_{q'})
	\end{equation}
\end{prevtheorem}

\begin{proof}
In view of \autoref{Cor:IsaacsOnePrime} the group  $C_q$ is a Sylow $q$-subgroup of $G$ and thus of $A$ as it is contained in $A$. The same corollary implies that  for any  $S \in \cart(G, X)$ we have $S_q \in \syl_q(G, X_q)$. But $S_q$ centralizes $S_{q'} \geq X_{q'}$ and thus $S_q \leq A= C_G(X_{q'})$ which in turn implies that $S_q  \in  \syl_q(A, X_q)$. 

 On the other hand, if $R \in \syl_q(A, X_q)$ then we claim that there exists a Carter subgroup $S \in \cart(G,X)$ having $R$ as its Sylow $q$-subgroup, i.e. $S_q=R$. Indeed, $R = (C_q)^a$ for some $a \in A$ because both $R$ and $C_q$ are Sylow $q$-subgroups of $A$. Note that $(C_{q'})^a$ also lies above $X_{q'}$ because $C_{q'}$ contains $X_{q'}$ and $a$ centralizes $X_{q'}$. Therefore, $C^a$ contains $X$ and clearly $C^a$ is a Carter subgroup of $G$. 
 Hence the desired Carter $S$ is $S= C^a$ for some $a \in A$.

We conclude therefore that the map sending any $S \in \cart(G, X)$ to  $S_q \in \syl_q(A, X_q)$ is well defined and  onto. We want now to count the number of  Carter subgroups of $G$ over $X$ that are  mapped to the same Sylow $q$-subgroup  of $A$ (over $X_q$). Assume $R \in \syl_q(A, X_q)$ and let $S \in \cart(G, X)$ with $S_q=R$. 
Then $R $ is also a Sylow $q$-subgroup of $G$ and 
by  \autoref{Cor:IsaacsOnePrime}, the Carter subgroups of $G$ that contain  $R$  are exactly
the subgroups $R U$, where $U$ is a Carter subgroup of $\cent_N(R)$. By
	assumption, $R$ contains $X_{q}$, so we see that $RU$ contains $X$ if
	and only if $U$ contains $X_{q'}$. It follows that the number $m(R)$ of Carter subgroups of $G$ that have  $R$ as a Sylow $q$-subgroup and  contain $X$ is
	\[
	m(R) = \card{\cart(G, R) \cap \cart(G, X)} =  n_C(\cent_N(R), X_{q'}).
	\]
	We argue that $n_C(\cent_N(R), X_{q'})=n_C(\cent_N(C_{q}), X_{q'})$ 
	and thus $m(R)$  is independent of the choice of the Sylow $q$-subgroup  $R \in  \syl_q(A, X_{q})$. 
	  Indeed, $R= S_q= (C_q)^a$  for some $a \in A= C_G(X_{q'})$. Therefore $X_{q'} = X_{q'}^a$ and $\cent_N(R)= \cent_N(C_q)^a$. So  
	\[
	n_C(\cent_N(R), X_{q'}) = n_C(\cent_N(C_{q})^a, X_{q'}^a)= 
	n_C(\cent_N(C_{q}), X_{q'}).
	\]

We see now that the total number of Carter subgroup of $G$ above $X$ equals 
\[
n_C(G,X) \, = \, \sum_{R \in \syl_{q}(A, X_{q})} n_C(\cent_N(R), X_{q'})= 
|\syl_q(A, X_q) | \cdot n_C(\cent_N(C_{q}), X_{q'}).
\]
The proof of the theorem is complete.
\end{proof}

 We would like to prove for some large class of soluble groups $G$ 
that $n_C(G,X)$ divides $(G:C)$, where $C$ is a Carter subgroup of $G$ and $X \leq C$. 
When this happens for all such subgroups $X$, 
we will say that $G$ satisfies property $(*)$.

\begin{corollary}\label{Cor:Asterisk}
Suppose that $G$ has a normal Hall $q'$-subgroup $N$ for some prime $q \in  \varpi(G)$, 
and assume that  $N$ and all its  subgroups satisfy $(*)$. 
Then $G$ satisfies $(*)$.
\end{corollary}

\begin{proofof}
Let  $X \leq G$ be a nilpotent subgroup of $G$ and $C \in \cart(G, X)$, we will show that $n_C(G,X)$ divides $(G:C)$.
By \autoref{Thm:EOnePrime}, we have 
\[
n_C(G,X) =|\syl_q(A, X_q) | \cdot n_C(\cent_N(C_{q}), X_{q'})
\]
where the notation is as in the theorem. Write $G=NC_q$   since (as we have seen)  $C_q \in \syl_q(G)$.
We claim that $C_{q'}$  is a Carter subgroup of $\cent_N(C_{q})$. 
To see this, note that if $x \in \N_{\cent_N(C_{q})} (C_{q'})$ then $x \in \N_G(C)  \cap C_N(C_q) = C  \cap C_N(C_q) = C_{q'}$ and thus the nilpotent subgroup  $C_{q'}$
of $\cent_N(C_{q})$ is self-normalising. Hence by assumption we have that 
$n_C(\cent_N(C_{q}), X_{q'})$
divides $(\cent_N(C_{q}) : C_{q'})$. In additon, Theorem 0.1 in \cite{arithmetical_questions} implies that $|\syl_q(A, X_q) |$ divides $|\syl_q(G) |= (N: \N_N(C_{q}))= (N: \cent_N(C_{q}))$.  

Hence $n_C(G,X)$ divides the product 
\[
(N: \cent_N(C_{q})) \cdot (\cent_N(C_{q}) : C_{q'}) = (N:C_{q'}).
\]

We conclude that $n_C(G,X)$ divides $(G:C)= (N:C_{q'})$, and  the proof is complete.

\end{proofof}

Recall that a group $G$ is said to possess a Sylow tower, 
or equivalently to be a Sylow tower group,
if there exists a normal series where each successive quotient is isomorphic to a Sylow subgroup of $G$.
Groups with a Sylow tower are, of course, soluble.
Moreover, they comprise a subgroup-closed formation 
and so quotients and subgroups of Sylow tower groups are again Sylow tower.

\begin{corollary}\label{Cor:SylowTower}
Groups with a Sylow tower satisfy $(*)$.
\end{corollary}

\begin{proofof}
We argue by induction on the order of the group,
the base case of the trivial group being trivially true.
Let $G$ be a group with a Sylow tower and suppose the
claim is valid for all groups with a Sylow tower and order
strictly less than $|G|$. In case $G=Q$ is a $q$-group the claim is clearly true, and thus we may assume that the order of $G$  is divisible by at least two primes.  
Then there exists a prime $q$ and a Hall $q'$-subgroup $N$
of $G$ such that $N$ is normal in $G$ and $N < G$. 
Clearly $N$ is a  group with a Sylow tower and thus, by induction, 
$N$ and all its subgroups satisfy $(*)$.
Therefore, applying \autoref{Cor:Asterisk} yields that $G$
satisfies $(*)$, completing the induction.
\end{proofof}

A direct consequence of \autoref{Cor:SylowTower}
is that supersoluble groups satisfy $(*)$. 

Even though we have tried to prove the analogue of \autoref{Thm:A} 
for Carter subgroups, we have not managed to get either a counterexample 
or a proof and so we state it here as a question.
\begin{question}\label{Conj:Carter}
Assume $G$ is a soluble group and $X\leq G$ is a nilpotent subgroup of $G$. 
Is it true that if  $n_C(G,X) \neq 0$ then  $n_C(G,X) \equiv 1\md{\widehat{m}_G}$, 
where 
\[
\widehat{m}_G:= \gcd\left\{p-1: p\in\mathbb{P},\, p \mid (G:C) \mbox{ with } C\in\cart(G)\right\}? 
\]
\end{question}  

{\bf Acknowledgment.} 
We thank the referee for valuable comments that led to cleaner and better presentation and 
for pointing out an error in the original manuscript.

\bibliographystyle{amsalpha}
\bibliography{Bibliography}

\end{document}